\documentclass[12pt]{amsart}

\usepackage[matrix,arrow,curve]{xy}
\usepackage{amsfonts,amssymb,mathrsfs}
\makeatletter \@addtoreset{equation}{section} \makeatother

\newcommand{\A}{\mathbb{A}}
\newcommand{\FF}{\mathbb{F}}

\newcommand{\CC}{\mathbb{C}}
\newcommand{\RR}{\mathbb{R}}
\newcommand{\QQ}{\mathbb{Q}}

\newcommand{\ZZ}{\mathbb{Z}}
\newcommand{\GG}{\mathbb{G}}
\newcommand{\PP}{\mathbb{P}}

\newcommand{\OOO}{{\mathscr{O}}} 
\newcommand{\NNN}{{\mathscr{N}}} 
\newcommand{\EEE}{{\mathscr{E}}}

\newcommand{\TTT}{{\mathscr{T}}}

\newcommand{\Pic}{\operatorname{Pic}}
\newcommand{\rk}{\operatorname{rk}}
\newcommand{\pr}{\operatorname{pr}}
\newcommand{\Sing}{\operatorname{Sing}}

\newcommand{\Gr}{\operatorname{Gr}}
\newcommand{\Eu}{\operatorname{Eu}}

\newcommand{\pt}{\operatorname{pt}}
\newcommand{\Aut}{\operatorname{Aut}}

\renewcommand{\emptyset}{\varnothing}

\newtheorem{theorem}[equation]{Theorem}
\newtheorem{proposition}[equation]{Proposition}
\newtheorem{corollary}[equation]{Corollary}
\newtheorem{lemma}[equation]{Lemma}

\theoremstyle{definition}

\newtheorem{notice}[equation]{}

\theoremstyle{remark}
\newtheorem{remark}[equation]{Remark}

\theoremstyle{remark}
\newtheorem{remarks}[equation]{Remarks}

\theoremstyle{definition}
\newtheorem{sit}[equation]{}

\title{Examples of cylindrical Fano fourfolds}

\author{Yuri Prokhorov}\thanks{The first author was
partially supported by the grant of
the Leading Scientific Schools No. 5139.2012.1, the Simons-IUM fellowship,
and the AG Laboratory SU-HSE, RF governmental
grant ag. 11.G34.31.0023.
}

\address{Yuri Prokhorov:
Steklov Mathematical Institute,
8 Gubkina street, Moscow 119991, Russia
\newline\indent
Department 
of Algebra, Faculty of Mathematics, Moscow State
University, Moscow 117234, Russia
\newline\indent
Laboratory of Algebraic Geometry, SU-HSE, 
7 Vavilova Str., Moscow 117312, Russia
}
\email{prokhoro@gmail.com}

\author{Mikhail Zaidenberg}

\address{Mikhail Zaidenberg: Universit\'e
Grenoble I, Institute Fourier, UMR 5582 CNRS-UJF, BP~74, 38402 Saint
Martin d'H\`eres cedex, France} \email{zaidenbe@ujf-grenoble.fr }

\begin{document}
\begin{abstract}
We construct 4 different families of smooth Fano fourfolds with Picard rank 1, 
which contain cylinders, i.e., Zariski open subsets of the form $Z\times\A^1$, 
where $Z$ is a quasiprojective variety. 
The affine cones over such a fourfold admit effective $\GG_{\operatorname{a}}$-actions.
Similar constructions of cylindrical Fano threefolds were done previously in 
\cite{Kishimoto-Prokhorov-Zaidenberg, Kishimoto-Prokhorov-Zaidenberg-2011}.
\end{abstract}
\maketitle

\section*{Introduction}
A smooth projective variety $V$ over $\CC$ 
is called {\em cylindrical} if it contains a cylinder, i.e., 
a principal Zariski open subset $U$ isomorphic to a product $Z\times\A^1$, 
where $Z$ is a variety and $\A^1$ stands for the affine line over $\CC$ 
(\cite{Kishimoto-Prokhorov-Zaidenberg, 
Kishimoto-Prokhorov-Zaidenberg-criterion}). 

Assuming that $\Pic (V)\simeq\ZZ$, 
we let $X$ be the affine cone over $V$.
Due to the criterion of 
\cite[Corollary 3.2]{Kishimoto-Prokhorov-Zaidenberg-criterion}, 
$X$ admits an effective action of the additive group 
$\mathbb{G}_{\operatorname{a}}$ 
if and only if $V$ is cylindrical. 
This explains our interest in cylindrical projective varieties.

 On the other hand, 
there is a well known Hirzebruch Problem of classifying all possible smooth 
completions of the affine space $\A^n$; see \cite{Furushima1990} 
and references therein 
for studies
on this problem. Similarly, it would be interesting 
to classify all cylindrical Fano varieties, or at least those 
with Picard number 1. 

The answer to the latter question is known in dimension 2, 
even without the restriction on the Picard number. 
Namely, a smooth del Pezzo surface of degree $d$ is cylindrical\footnote{If 
$\Pic(V)\not\simeq\ZZ$, by cylindricity we mean here the polar cylindricity, 
see \cite{Kishimoto-Prokhorov-Zaidenberg-del-pezzo}.} 
if and only if 
$d\ge 4$ \cite{Kishimoto-Prokhorov-Zaidenberg, 
Kishimoto-Prokhorov-Zaidenberg-del-pezzo, Cheltsov-Park-Won-2013, 
Cheltsov-2014}.

In dimension 3, a cylindrical Fano threefold 
must be rational, see Remark \ref{rem:rationality}(a) below. However, 
even for cylindrical 
Fano threefolds with Picard 
number 1, the 
complete classification is lacking. Certain classes of such 
threefolds were described in \cite{Kishimoto-Prokhorov-Zaidenberg} and 
\cite{Kishimoto-Prokhorov-Zaidenberg-2011}; we provide below
 the exhaustive list of examples known to us. 

The projective space $\PP^3$ and the smooth quadric in $\PP^4$ are 
cylindrical, since they contain a Zariski open subset isomorphic to the affine 
space $\A^3$. By the same reason, the Fano threefold of index $2$ and degree 
$5$ is also 
cylindrical. A smooth intersections of two quadrics in $\PP^5$ is always 
cylindrical (see \cite[Propositions 5.0.1]{Kishimoto-Prokhorov-Zaidenberg}). 
The same is true for the Fano threefolds of index 1 and genus $12$ 
(\cite[Propositions 5.0.2]{Kishimoto-Prokhorov-Zaidenberg}). 
The moduli space of the latter family has
dimension $6$, while the subfamily of completions of $\A^3$ 
is four-dimensional.
There are two more families of cylindrical Fano threefolds with Picard 
number 1, index 1 and genera $g = 9$ and $10$. 
These fill in at least hypersurfaces in the corresponding moduli spaces 
(\cite{Kishimoto-Prokhorov-Zaidenberg-2011}). 
We do not dispose a single example of a smooth 
Fano threefold of genus $7$ carrying a cylinder.

In this paper we construct examples of smooth, cylindrical 
Fano fourfolds with Picard rank 1. 
Let us recall first the standard terminology and notation.
Given a smooth Fano fourfolds $V$ with Picard rank 1, the \textit{index} $r$ of $V$ 
is the integer $r$ such that $-K_V\in r[H]$, where $[H]$ is the ample divisor class generating 
the Picard group: $\Pic (V)=\ZZ\cdot [H]$.
The \textit{degree} $d=\deg V$ of $X$ is the degree with respect to $H$. 
It is known that $1\le r\le 5$. Moreover,
if $r=5$ then
$V\simeq \PP^4$ , and if $r=4$ then $V$ is a quadric in 
$\PP^5$. Smooth Fano fourfolds of index $r=3$ are called {\em del 
Pezzo fourfolds}; 
their degrees vary in the range $1\le d\le 5$ (\cite{Fujita1980-84}). 
Smooth Fano fourfolds of index $r=2$ are called {\em Mukai fourfolds}; 
their degrees are even and can be written as $d=2g-2$, where $g$ is called the \textit{genus} of $V$.
The genera of Mukai fourfolds satisfy 
$2\le g\le 10$ (\cite{Mukai-1989}). The classification of Fano fourfolds of
index $r=1$ is not known. 

We let below $V_d$ 
be a Mukai fourfold of degree $d$, and $W_d$ be a del 
Pezzo fourfold of degree $d$. 

Our main results are summarized in the following theorem.

\begin{theorem}\label{thm:main thm} 
There are the following families of cylindrical, smooth, 
rational Fano fourfolds:

\begin{enumerate} 

\item the smooth intersections $V_{2\cdot 2}$ of 
two quadrics in $\PP^6$;

\item the del Pezzo fourfolds $W_5$ of degree 
$5$;

\item certain Mukai fourfolds $V_{14}$ of genus $8$ 
varying over a a family of codimension $1$ in the moduli space; 

\item certain Mukai fourfolds $V_{12}$ of genus $7$ 
varying over a family of codimension $2$ \textup(of dimension $13$\textup) in the moduli space. 
\end{enumerate} 
\end{theorem}

The proof exploits explicit constructions of the Fano fourfolds as in (i)--(iv)
via Sarkisov links. These constructions are borrowed in \cite{Prokhorov-1993c}. 
However, we recover some important details that are just sketched in 
\cite{Prokhorov-1993c}. The proof proceeds as follows. Starting with 
a (simpler) pair $(V,D)$, where $V$ is a smooth Fano fourfold and $D$ 
an effective divisor 
on $V$ such that $V\setminus D$ 
contains a cylinder, we reconstruct it into a new (more complex) pair $(V',D')$ via 
a Sarkisov link that does not destroy the cylinder. 

\begin{remark}\label{rem:rationality} (a) It is known that the del Pezzo 
fourfold 
$W_5$ in (ii) is unique up to isomorphism. In \cite{Prokhorov1994} 
all smooth del Pezzo fourfold completions $(W_5, A)$ of $\A^4$ 
by an irreducible 
divisor $A$ were described. 
Up to isomorphism of pairs, there are exactly 4 such completions. 
In contrast, we show that for an ample divisor $H$
generating 
the Picard group $\Pic (W_5)$, the complement 
$W_5\setminus H$ contains a cylinder, provided 
the hyperplane section $H$ is singular.

(b) For some families of smooth, rational Fano fourfolds constructed in 
\cite{Prokhorov-1993c},
the existence of cylindrical members remains unknown. 
For instance, this concerns the smooth, rational cubic fourfolds in $\PP^5$. 

(c) 
We do not know whether all the cylindrical Fano 
fourfolds are rational, while the varieties in Theorem \ref{thm:main thm} 
are. 
However, any cylindrical Fano threefold is rational.
Indeed, 
for any smooth Fano variety $V$ the plurigenera and irregularity
vanish. If $V$ contains a cylinder $Z\times\A^1$, 
then, for $Z$, the plurigenera and irregularity vanish too. 
If $Z$ is a surface, then it must be rational due to the Castelnuovo rationality 
criterion.
Of course, this argument fails in higher dimensions.
 \end{remark} 

The content of the paper is as follows. In section \ref{sec:prelim} we give some technical preliminaries. 
 We prove items (i), (ii), (iii), and (iv) of 
Theorem \ref{thm:main thm} in subsequent sections \ref{section-2.2}, \ref{sec:W5}, \ref{sec:g=8}, 
and \ref{sec:g=7}, 
more precisely, 
in Theorems \ref{cor:cylinder in W22}, \ref{theorem-w5}, \ref{thm:cylinder-in-V14}, 
and \ref{thm:cylinders-in-V12}, respectively.

\medskip

{\em Acknowledgments}.
This work was done during a stay of the first author at the Institute Fourier,
Grenoble. The authors thank this institution for its hospitality and a generous support.

\section{Preliminaries}\label{sec:prelim}
\begin{notice}{\bf Notation.} Let $X$ be a smooth projective variety 
and $C\subset X$ be a smooth subvariety. We denote by
\begin{enumerate}
\item[] 
 $K_X$ the canonical divisor of $X$;
 \item[] 
 $\TTT_X$ the tangent bundle of $X$;
 \item[] 
 $\NNN_{C/X}$ the normal bundle of $C$ in $X$;
 \item[] 
 $c_i(\EEE)$ the $i$th Chern class of a vector bundle $\EEE$ on $X$, 
 and \item[] $c_i(X)=c_i(\TTT_X)$;
\item[] $\Eu(X)$ the Euler number of $X$.
 \end{enumerate}
\end{notice}

As for the Schubert calculus, we follow the notation in \cite{Griffiths-Harris-1978}.

In this section we gather some auxiliary facts that we need in the subsequent sections. 
The following lemma is a special case of the Riemann-Roch Theorem. 

\begin{lemma}
For a smooth projective fourfold $X$
and a divisor $D$ on $X$ we have
\begin{multline*}
\chi(\OOO_X(D)) = \frac1{24}\left[D^4 + 2D^3\cdot c_1(X) +
D^2\cdot (c_1 (X)^2 + c_2 (X))+\right.\\\left. + D\cdot c_1
(X)\cdot c_2 (X)\right]+ \chi(\OOO_X).
\end{multline*}
\end{lemma}

The next lemma can be deduced, for instance, from \cite[Th. 15.4]{Fulton1998}.

\begin{lemma}\label{lem:2.3}
Let $X$ be a smooth projective fourfold,
and $\rho\colon \widetilde{X}\to 
X$ be a blowup with a smooth center 
$C \subset X$ and the exceptional divisor $E=\rho^{-1}(C)$.
\begin{enumerate}
 \item 
If $C$ is a curve of genus $g(C)$, then
\[
c_2(\widetilde X)=\rho^*c_2(X)+ (6g(C)-6 -K_X\cdot C) A,
\]
where $A$ is the class of the fiber of $E\to C$. 
 \item 
If $C$ is a surface, then 
\[
c_2(\widetilde X)=\rho^*c_2(X)+\rho^*C+\rho^*K_X\cdot E.
\] 
\end{enumerate}
\end{lemma}

\begin{lemma}
\label{lemma-intersection-theory} 
In the notation of Lemma \textup{\ref{lem:2.3}}, 
for a divisor $H$ on $X$
the following relations in the Chow ring $A (\widetilde X)$ hold:
\begin{enumerate}
 \item 
if $C$ is a curve, then 
\[
\begin{array}{l}
(\rho^* H)^{4}= H^4,\quad (\rho^* H)^{3}\cdot E=(\rho^* H)^{2}\cdot E^2= 0,
\\[7pt]
\rho^* H\cdot E^3 = H^3\cdot C,
\quad 
E^4=-K_X \cdot C+2g(C)-2\,;
\end{array}
\]
 \item 
if $C$ is a surface, then 
\[
\begin{array}{l}
(\rho^* H)^{4}= H^4,\quad (\rho^* H)^{3}\cdot E=0,\quad (\rho^* H)^{2}\cdot E^2= -C\cdot
H^2,
\\[7pt]
\rho^* H\cdot E^3 =-H|_ C\cdot K_C+K_X\cdot H\cdot C,
\\[7pt]
E^4=c_2(X)\cdot C+K_X|_C\cdot K_C-c_2(C)-K_X^2\cdot C.
\end{array}
\]
\end{enumerate}
\end{lemma}
\begin{proof} 
The proof is straightforward. It uses the projection formula,
the identification $E=\PP_C(\NNN_{C/X}^*)$, and the equality
$E|_E= c_1(\OOO_{\PP_C(\NNN_{C/X}^*)}(-1))$ in the Chow ring $A(E)$, 
see e.g. \cite{Fulton1998}.
\end{proof}

\begin{lemma}\label{lem:smooth section}
Let $X\subset\PP^N$ be a smooth projective variety
of dimension $n$. Assume that $X$ contains a $k$-dimensional 
linear subspace $\Lambda$, and let $H$ be a general hyperplane section
of $X$ containing $\Lambda$. If $n>2k$, then $H$ is smooth.
\end{lemma}
\begin{proof}
By Bertini theorem, $H$ is smooth outside $\Lambda$.
We use a parameters count. The set $\mathcal{H}$ of all hyperplane sections of 
$X\subset\PP^N$
that are singular at some point of $\Lambda$ is Zariski closed in ${\PP^N}^{*}$,
of dimension 
\[\dim\mathcal{H}
\le k+ N-(n+1)< N-(k+1). 
\]
On the other hand, 
the set of all hyperplane sections containing $\Lambda$ has dimension $N-(k+1)$. 
Now the result follows.
\end{proof}

\begin{lemma}
Let $W=W_d\subset \PP^{d+2}$ be a del Pezzo fourfold of degree $d\ge 3$.
Then $W$ contains a line. Given a line $l\subset W$, 
one of the following holds:
\begin{equation}\label{equation-normal-bundle}
\begin{aligned}
\NNN_{l/W}\simeq \OOO_{\PP^1}\oplus \OOO_{\PP^1}\oplus \OOO_{\PP^1}(1),& 
\qquad (a)
\\
\NNN_{l/W}\simeq \OOO_{\PP^1}(-1)\oplus \OOO_{\PP^1}(1)\oplus \OOO_{\PP^1}(1).& 
\qquad (b)
\end{aligned}
\end{equation}
\end{lemma}

\begin{proof} Let $H_1, H_2$ be two generic hyperplane sections of $W$. 
Then $F= H_1\cap H_2 $ is a smooth del Pezzo surface of degree $d$ 
in $W$ 
anticanonically embedded in $\PP^{d+2}$, which contains some lines. 
Let $l$ be a line on $F$, and consider
 the exact sequence
$$0\to \NNN_{l/F}\to \NNN_{l/W}\to \NNN_{F/W}\vert_l\to 0\,.$$ 
We have
$$\NNN_{l/F}\simeq\OOO_l(1),\ \NNN_{F/W}\vert_l\simeq 
\OOO_W(1)\vert_l\oplus \OOO_W(1)\vert_l,\ 
\NNN_{l/W}\simeq \OOO_l(a)
\oplus \OOO_l(b)\oplus \OOO_l(c)\,,$$
where $a+b+c=1$. Furthermore, $a,b,c\le 1$ since $\operatorname{Hom} 
(\OOO_l(x), \OOO_l(y))=0$ for $x>y$. 
It follows also that $a,b,c\ge -1$. Assuming that $a\le b\le c$ 
we obtain that
$(a,b,c)\in\{(0,0,1),(-1,1,1)\}$, as stated. 
\end{proof} 

\begin{corollary}\label{cor:2.8}
Let $W=W_d\subset \PP^{d+2}$ be a del Pezzo fourfold of degree $d\ge 3$.
Then the Hilbert scheme $\mathfrak L(W)$ of lines on $W$ 
is reduced, nonsingular, and $\dim\mathfrak L(W)=4$. 
Through any point $P\in W$ passes a family of lines on $W$ 
of dimension $\ge 1$.
For any $d\ge 4$ this family 
has dimension $1$.
\end{corollary}
\begin{proof}
Since $H^1(\PP^1, \NNN_{l/W}) =0$, 
using deformation theory we obtain that
$\mathfrak L(W)$ is reduced, nonsingular and 
$\dim \mathfrak L(W) = \dim H^0(\PP^1, \NNN_{l/W}) =4$. 

If $d=3$, then $W$ is a smooth cubic in $\PP^5$.
The tangent section $T$ at a point $P$ of $W$ 
is a singular cubic threefold in the projective tangent space 
$\overline {T_P}W\simeq \PP^4$. 
In an affine chart in $\PP^4$ centered at $P$, the equation of $T$ has just
the quadric and cubic homogeneous terms. The common zeros of these 
form a cone swept out by a family of lines through $P$ 
of dimension $\ge 1$.

Let further $d\ge 4$, and so $\operatorname{codim}_{\PP^{d+2}} W\ge 2$. 
Then the tangent section $T=W\cap \overline {T_P} W$ 
cannot be a divisor in $W$. Indeed, otherwise this should be 
a hyperplane section generating the Picard group $\Pic (W)$. 
However, for a generic point $P'\in W\backslash \overline {T_P}W$
there is a hyperplane $H\subset \PP^{d+2}$ through $P'$, which
contains $\overline {T_P}W$. This leads to a contradiction. 

 Any line through $P$ in $W$ is contained in the tangent space 
$\overline {T_P}W$, hence also in $T$, which is at most two-dimensional. 
Thus the family of lines in $W$ through $P$ is at most one-dimensional. 
At a generic point of $W$, it should be one-dimensional, because 
$\dim \mathfrak L(W)=4$. Therefore for any point $P\in W$, this family is of
dimension 1. 
\end{proof}

\section{Intersection of two quadrics in $ \PP^6$}
\label{section-2.2}
It is known that any smooth intersection of two quadrics in $\PP^5$ 
contains a cylinder \cite[Prop. 5.0.1]{Kishimoto-Prokhorov-Zaidenberg}. 
Similarly, this holds for any smooth intersection $W_{2\cdot 2}$
of two quadrics in $\PP^6$, see Theorem \ref{cor:cylinder in W22}. 
The proof is based on 
a standard
construction of $W_{2\cdot 2}$ via 
a Sarkisov link, see \eqref{equation-sarkisov-link-2.2}. 
For the reader's convenience,
we recall this construction and 
some of its specific properties used in Sect. \ref{section-g=12} below. 

\begin{proposition}\label{sarkisov-link-2.2}
Let as before $W=W_{2\cdot 2}$ be a smooth 
intersection of two quadrics in $\PP^6$,
and let $H$ be the ample generator of $\Pic(W)$ \textup(the class of 
a hyperplane section\textup).
 Given a line $l\subset W$,
the projection with center $l$ is a birational map $\phi: W \dashrightarrow \PP^{4}$,
which fits in the diagram 
\begin{equation}
\label{equation-sarkisov-link-2.2}
\xymatrix{
&\widetilde W\ar[dr]^{\varphi}\ar[dl]_{\rho}&
\\ 
W\ar@{-->}[rr]^{\phi}&&\PP^4
} 
\end{equation}
where $\rho$ is the blowup of $l$ with exceptional divisor $E$, and $\varphi$ is 
a birational morphism defined by the linear system $|\rho^*H-E|$. Furthermore, the following hold.
\begin{enumerate}
\item 
The $\varphi$-exceptional locus is an irreducible divisor $D\subset \widetilde W$. 
\item 
Letting $L$ be the ample generator of $\Pic(\PP^4)$
we obtain 
\[
\varphi^* L\sim \rho^*H-E,\quad
D\sim 2\rho^*H-3E,\quad
\rho^* H\sim 3\varphi^* L-D,\quad 
\rho^* E\sim 2\varphi^* L-D.
\]
\item
The divisor $\rho(D)$ is swept out by lines meeting $l$.
The image $F=\varphi(D) $ is a surface in $\PP^4$ of degree $5$ 
with at worst isolated singularities. The singularities of 
$F$ are the $\phi$-images of planes in $W$ containing $l$. 
For a general line $l\subset W$ the surface $F$ is smooth,
and $\varphi$ is the blowup with center $F$.
\item
$W\setminus \rho(D)\simeq \PP^4 \setminus \varphi(E)$,
where $\rho(D)$ is cut out on $W$ by a quadric in $\PP^6$, and $\varphi(E)$ 
is a quadric in $\PP^4$. 
\item
The quadric $\varphi(E)\subset \PP^4$ is singular, and 
its singular locus coincides with the locus of points $P\in\varphi(E)$, 
such that the restriction 
$\varphi|_E: E\to \varphi(E)$ is 
not an isomorphism over $P$.
\end{enumerate}
\end{proposition} 

\begin{proof}
Since $l$ is a scheme-theoretic intersection of members of $|H-l|$,
the linear system $|\rho^*H-E|$ is base points free.
Hence the divisor $-K_{\widetilde W}=\rho^*H+2(\rho^*H-E)$ is ample, i.e. 
$\widetilde W$ is a Fano fourfold with $\rk \Pic(\widetilde W)=2$.
By the Cone Theorem there exists a Mori contraction $\varphi: \widetilde W\to U$ different from $\rho$.
If $ \widetilde C\subset \widetilde W$ is the proper transform of a line $C\subset W$ meeting $l$,
then $(\rho^*H-E)\cdot \widetilde C=0$. In particular, the divisor $\rho^*H-E$ is not ample.
So, $\rho^*H-E$ yields a supporting linear function for the extremal ray generated 
by the curves contained in the
fibers of $\varphi$. Furthermore, we can write $\rho^*H-E=\varphi^* L$, where $L$ is 
the ample generator of 
$\Pic(U)\simeq \ZZ$. 
We have $L^4=(\rho^*H-E)^4=1$.
By the Riemann-Roch and Kodaira Vanishing Theorems we have
$\dim |\rho^*H-E | =4$.
Therefore, $|\rho^*H-E |$ defines a birational
morphism $\widetilde W\to \PP^4$, which coincides actually 
with the map $\varphi$.

Since $\dim |2\rho^*H-3E | =0$ and $(\rho^*H-E)^3\cdot (2\rho^*H-3E)=0$,
the linear system $|2\rho^*H-3E |$ contains a unique divisor $D$ 
contracted by $\varphi$. Since $\rk \Pic (\widetilde W)=2$, the divisor
$D$ is irreducible. This proves (i) and (ii).

Besides, $L^2\cdot D^2=-5$, and so the image $F=\varphi(D)$ is a quintic surface in $\PP^4$.
Since $\phi$ is a projection and $W$ is an intersection of quadrics, 
any positive dimensional fiber of the birational morphism $\varphi$ is the proper 
transform either of a line meeting $l$,
or of a plane containing $l$. If there is no plane in $W$ containing $l$,
then any 
fiber of $\varphi: D\to F$ has dimension at most $1$. 
In this case $F$ is smooth, and $\varphi$ is the blowup of $F$
(see \cite{Ando1985}). 
In general, there is a finite number of two-dimensional fibers, 
and $F$ is singular at the corresponding points.
The local structure of $\varphi$ near ``bad'' fibers is described in \cite{Andreatta1998a}. This proves (iii).

The assertion of (iv) follows from our construction and (ii).

(v) The quadric $\varphi(E)$ is singular, since it contains the surface $F$
of degree $5$. 
Since $-K_E=3(\rho^*H-E)$, the restriction 
$\varphi|_E: E\to \varphi(E)$ is a crepant morphism.
\end{proof}

\begin{remark}
Any smooth intersection of two quadrics $W=W_{2\cdot 2}\subset \PP^6$
contains exactly $64$ planes, and the classes of these planes span the 
cohomology group $H^4(W, \ZZ)$ \cite{Reid1972}. 
Hence the case of a singular surface $F=\varphi(D)$
does occur.
\end{remark}

\begin{remarks}\label{remark-singularities-quadric} 1. The singular locus of the quadric $\varphi(E)$ is a point
if and only if $\NNN_{l/W}$ is of type (a) in \eqref{equation-normal-bundle},
otherwise 
this is a line.

2. Note that the divisor $E$ is $\varphi$-ample. 
Hence $E$ meets any nontrivial fiber of $\varphi$. Furthermore,
it meets any two-dimensional fiber along a subvariety of positive dimension.
Thus, if $l$ is contained in a plane $\Pi\subset W$, then $\varphi(E)$
must be singular at $\phi(\Pi)$.
If $l$ is contained in two planes $\Pi_1,\, \Pi_2\subset W$, then 
the quadric $\varphi(E)$ has two distinct singular points $\phi(\Pi_1)$
and $\phi(\Pi_2)$. Since the singular locus of a quadric is
a linear subspace, in the latter case $\Sing(\varphi(E))$ contains a line.
\end{remarks}

The next lemma is a simple observation. 

\begin{lemma}\label{lem:cylinder-P4}
Let $Q\subset \PP^4$ be a quadric and $L\subset \PP^4$ be a hyperplane.
Suppose that $\Sing (Q) \cap L$ contains a point $P$. 
Then the projection $\PP^4\dashrightarrow\PP^3$ with center $P$ defines a cylinder structure in 
$\PP^4 \setminus (Q\cup L)$.
\end{lemma}

The following is the main result of this section.

\begin{theorem}\label{cor:cylinder in W22} 
In the notation of Proposition \textup{\ref{sarkisov-link-2.2}},
the Zariski open set $W\setminus \rho(D)$ contains a cylinder. 
\end{theorem}

\begin{proof} 
Indeed, $\PP^4\setminus \varphi(E)$ is a cylinder because 
$\varphi(E)$ is a singular quadric. Now the result follows due to (iv) 
in Proposition \ref{sarkisov-link-2.2}.
\end{proof}

The next lemma and its corollary will be used in section \ref{section-g=12}.

\begin{lemma}\label{lemma-cylinder-W2.2}
 In the notation as before, let $H_0\subset W$ be a hyperplane 
section containing $l$.
Then the image $L_0=\phi_*(H_0)$ is a hyperplane in $\PP^4$, and
\begin{equation}\label{eq:isomm}
W\setminus (\rho(D)\cup H_0)\simeq \PP^4 \setminus 
(\varphi(E)\cup L_0).
\end{equation}
If $\Sing(\varphi(E))\cap L_0\neq \emptyset$, then $W\setminus 
(\rho(D)\cup H_0)$
contains a cylinder. 
Assume that there is a plane $\Pi\subset W$ such that $l\subset 
\Pi\subset H_0$. 
Then the condition $\Sing(\varphi(E))\cap L_0\neq \emptyset$
holds. 
\end{lemma}
\begin{proof} 
Let $\widetilde H_0\subset \widetilde W$ be the proper transform of $H_0$.
We may write $\rho^*H_0=\widetilde H_0+kE$ for some $k\ge 1$.
Since $(\rho^*H-E)^3\cdot \widetilde H_0=3-2k$, we have $k=1$ and $\deg \varphi(\widetilde H_0)=1$.
Hence $L_0$ is a hyperplane. The existence of an isomorphism in (\ref{eq:isomm}) follows by our construction.
The projection of $\PP^4$ from a point $P\in \Sing(\varphi(E))\cap L_0$
produces a cylinder structure on $\PP^4 \setminus (\varphi(E)\cup L_0)$ (see Lemma 
\ref{lem:cylinder-P4}).
Finally, $\phi(\Pi)$ is a point contained in both $L_0$ and $\Sing(\varphi(E))$
(see Remark \ref{remark-singularities-quadric}).
\end{proof}

The following corollary is immediate.

\begin{corollary}\label{corollary-cylinder-W2.2}
Let $H_0\subset W$ be a hyperplane section.
Assume that 
there exists a plane $\Pi\subset W$ such that 
$\dim \Pi\cap H_0\ge 1$. Then $W\setminus H_0$
contains a cylinder. 
\end{corollary}

\section{The quintic del Pezzo fourfold}\label{sec:W5}
According to \cite[II]{Fujita1980-84} a del Pezzo fourfold of degree $5$ 
is unique up to isomorphism. It can be realized as 
a smooth section $W=W_5\subset \PP^7$ of the Grassmannian 
$\Gr(2, 5)$ under its Pl\"ucker embedding in $\PP^9$ by a codimension 2 linear subspace. 
Clearly, the class of a hyperplane section $H$ generates the group $\Pic(W)$,
and $-K_W\sim 3H$. The variety $W$ is an intersection of quadrics (see 
\cite[ch. 1, \S 5]{Griffiths-Harris-1978}).

The following theorem is the main result of this section.
\begin{theorem}
\label{theorem-w5}
Let $W=W_5\subset \PP^7 $ be a del Pezzo fourfold of
degree $5$,
and let $A$ be a singular hyperplane section of $W$.
Then the Zariski open set $W\setminus A$ contains a cylinder. 
\end{theorem}
The proof is done at the end of the section. 
First we need some preliminaries. 

There are two types of planes in
$\Gr(2, 5)$, namely, the Schubert varieties $\sigma_{3,1}$ and 
$\sigma_{2,2}$ (\cite[ch. 1, \S 5]{Griffiths-Harris-1978}), where
\begin{itemize}
 \item 
$\sigma_{3,1}:=\{l\in \Gr(2, 5)\,|\, p\in l \subset \PP^3\}$ with $\PP^3\subset \PP^4$ a fixed
$3$-dimensional subspace and $p\in\PP^3$ a fixed point;
\item 
$\sigma_{2,2}:=\{l\in \Gr(2, 5)\,|\, l \subset \PP^2\}$ with $\PP^2\subset \PP^4$ a fixed 
plane.
\end{itemize}

\begin{lemma}
\label{Lemma-2.4.}
For the Chern classes of $G=\Gr(2,5)$ we have
\[
\begin{array}{l@{\qquad}l}
c_1(G)= 5\sigma_{1,0}, &
c_2(G) =11\sigma_{2,0}+12\sigma_{1,1},
\\[5pt]
c_3(G) = 15\sigma_{3,0}+30\sigma_{2,1},
&
c_4(G) = 35\sigma_{3,1}+25\sigma_{2,2}.
\end{array}
\]
\end{lemma}

\begin{proof}
Let $\mathcal I\to G$ be the universal subbundle and $\mathcal Q\to G$ 
the universal factor-bundle 
(see \cite[ch. 3, \S 11]{Griffiths-Harris-1978}).
Then $c_1(\mathcal I)=\sigma_{1,0}$, $c_2(\mathcal I)=\sigma_{1,1}$, and
$c_r(\mathcal Q)\sim \sigma_{r,0}$ ({\em ibid}).
Since $T_G \simeq \mathcal I^*\otimes \mathcal Q$, 
standard computations give
the result
(see e.g. \cite[Example 14.5.2]{Fulton1998}).
\end{proof}

\begin{corollary}\label{Chern-classes-W5} For $W=W_5$ we have
\[
\text{
$c_1(W)=3\sigma_{1,0}|_W$,\quad
$c_2(W)= 4\sigma_{2,0}|_W+5\sigma_{1,1}|_W$, and
$\Eu(W)=6$.
}
\]
\end{corollary}

\begin{proof} The Adjunction Formula for the total Chern classes yields:
$$c_t(G)=c_t(W)\cdot c_t(\NNN_{W/G})=c_t(W)\cdot (1+\sigma_{1,0})^2\,.$$
Inversing the last factor gives the desired equalities. \end{proof}

 The following proposition proven in \cite{Todd1930} (see also \cite[3.3]{Debarre2012}) 
deals with the planes in the fourfold $W_5$. 

\begin{proposition}
\label{Proposition-2.2.}
Let $W=W_5\subset \PP^7 $ be a Fano fourfold of index $3$ and degree $5$.
Then the following hold. 
\begin{enumerate}
\item 
$W$ contains exactly one $\sigma_{2,2}$-plane $\Xi$ and a one-parameter 
family of $\sigma_{3,1}$-planes. 

\item 
Any $\sigma_{3,1}$-plane $\Pi$ meets $\Xi$ along a tangent line to a fixed conic
$C\subset \Xi$. 

\item 
Any two $\sigma_{3,1}$-planes $\Pi_1$ and $\Pi_2$ meet at a point $p\subset \Xi \setminus C$. 

\item 
Let $R$ be the union of all $\sigma_{3,1}$-planes on $W$. 
Then $R$ is a hyperplane section of $W$ and $\Sing(R) = \Xi$. 
\end{enumerate}
\end{proposition}

\begin{remark}
In \cite[II, \S 10]{Fujita1980-84} $\sigma_{3,1}$-planes ($\sigma_{2,2}$-planes, respectively) are called planes 
of {\em vertex type} (of {\em non-vertex type}, respectively).
\end{remark}

The following lemma completes the picture. 

\begin{lemma}\label{lemma-W5-normal-bundle}
Let $\Lambda\subset W$ be a plane.
Then $c_1(\NNN_{\Lambda/W})=0$ and
$c_2(\NNN_{\Lambda/W})=2$ \textup($c_2(\NNN_{\Lambda/W})=1$, respectively\textup)
if $\Lambda$ is of type $\sigma_{2,2}$ \textup(of type $\sigma_{3,1}$, respectively\textup).
\end{lemma}

\begin{proof}
Let $l$ be the class of a line on $\Lambda$.
By Corollary \ref{Chern-classes-W5} we have
$c_1(W)|_\Lambda=3\sigma_{1,0}|_\Lambda=3l$ and
$c_2(W)|_\Lambda= 4\sigma_{2,0}|_\Lambda+5\sigma_{1,1}|_\Lambda$.
Since $c_t(\NNN_{\Lambda/W})\cdot c_t(\Lambda) =c_t(W)|_\Lambda$, we have
\begin{eqnarray*}
c_1(\NNN_{\Lambda/W})&=&c_1(W)|_\Lambda-c_1(\Lambda)=0,
\\
c_2(\NNN_{\Lambda/W})&=& c_2(W)\cdot\Lambda-c_1(\NNN_{\Lambda/W})\cdot c_1(\Lambda)-c_2(\Lambda)=
\\
&=&4\sigma_{2,0}\cdot\Lambda+5\sigma_{1,1}\cdot\Lambda-3.
\end{eqnarray*}
These lead to the desired equalities.
\end{proof}

\begin{corollary}
The groups $H^q(W,\ZZ)$ vanish if $q$ is odd,
$H^2(W,\ZZ)\simeq H^6(W,\ZZ)\simeq \ZZ$, and $\rk H^4(W,\ZZ)=2$.
Moreover, $H^4(W,\ZZ)/\operatorname{Tors}$ is generated 
by the classes of a $\sigma_{3,1}$-plane $\Pi$
and the $\sigma_{2,2}$-plane $\Xi$.
\end{corollary}
\begin{proof}
The first two statements follow by the Lefschetz Hyperplane Section Theorem.
Then by Corollary \ref{Chern-classes-W5} we have $\rk H^4(W,\ZZ)=2$.
By Lemma \ref{lemma-W5-normal-bundle} $\Pi^2=1$ and $\Xi^2=2$.
Furthermore, $\sigma_{1,0}^2|_W\sim 2\Xi+3\Pi$ and so $\Pi\cdot \Xi=-1$.
Hence the intersection matrix of $\Xi$ and $\Pi$ is unimodular. Now
the last assertion follows by the Poincar\'e duality.
\end{proof}

\begin{lemma}\label{lemma-W5-plane}
If $A$ is a singular member of the linear system $|H|$, then the threefold $A$ contains a plane.
\end{lemma}
\begin{proof}
If the singular locus $\Sing(A)$ is of positive dimension, then 
it meets $R$. Pick a point $P\in \Sing(A)\cap R$.
There is a $\sigma_{3,1}$-plane $\Pi\subset R$ passing through $P$.
Since $P\in \Sing(A)$ and $A$ is a hyperplane section, $\Pi\subset A$.
Thus we may assume that $A$ has isolated singularities.
Let $P\in \Sing(A)$. Take a general member $H\in |H|$ passing through $P$.
By Bertini's theorem $H\cap A$ is an irreducible surface with $\Sing(H\cap A)=\{P\}$. 
Clearly, $H\cap A$ is not a cone for a general choice of $H\in |H|$.
Then by \cite{Hidaka1981} $H\cap A$ is a del Pezzo surface of degree $5$ 
with only Du Val singularities. In particular, the singularities of $A$ are
terminal and Gorenstein \cite[Cor. 5.38]{Kollar-Mori-1988}.
Thus $A$ is a Fano threefold of index $2$ and degree $5$ with terminal Gorenstein singularities.
There are 3 classes of such varieties \cite[Cor. 8.7]{Prokhorov-GFano-1} (see also 
\cite[Th. 2.9]{Fujita-1986-1}, \cite{Todd1930}).
As follows from the classification, in any case $A$ contains a plane.
\end{proof}

We construct below a cylinder in 
$W\setminus A$. This is done in Proposition \ref{Proposition-3.2} 
and Corollary \ref{Corollary-3.3.} in the case that $A$ contains a $\sigma_{2,2}$-plane, 
and in Proposition \ref{Proposition 3.8.} 
and Corollary \ref{Corollary-3.3a.} in the case that $A$ contains a $\sigma_{3,1}$-plane.

\begin{proposition} [{\cite[II, (7.8)]{Fujita1980-84}, {\cite{Prokhorov1994}}, 
\cite[3.6]{Debarre2012}}] 
\label{Proposition-3.2}
Let $\Xi\subset W$ be the $\sigma_{2,2}$-plane. 
Then there is a commutative diagram 
\[
\xymatrix{
&\widetilde W\ar[dr]^{\varphi}\ar[dl]_{\rho}&
\\
W\ar@{-->}[rr]^{\phi} &&\PP^4 
}
\]
where 
\begin{enumerate} 
\item 
$\rho : \widetilde W\longrightarrow W$ is the blow-up of $\Xi$, 
$\phi : W\dashrightarrow \PP^4$ is the projection from $\Xi$, $\varphi : W\longrightarrow \PP^4$ is 
the blowup of a rational normal cubic curve $Y\subset \PP^4$; 

\item 
$\varphi :\widetilde W \longrightarrow \PP^4$ is defined by the linear system $|\rho^* H - E|$, 
where $E =\rho^{-1}(\Xi)$ 
is the exceptional divisor; 

\item 
$\varphi(E) = \PP^3 = \langle Y\rangle$
is the linear span of $Y$; 

\item 
the exceptional divisor $\widetilde R=\varphi^{-1}(Y)$ of $\varphi$ coincides 
with the proper 
transform of $R$ in $\widetilde W$, and $\widetilde R\sim \rho^*H - 2E$ 
on $\widetilde W$. 
\end{enumerate}
\end{proposition}

\begin{proof}[Sketch of the proof.] 
Using Lemmas \ref{lemma-intersection-theory} and 
\ref{lemma-W5-normal-bundle} it is easy to deduce that $(H^* - E)^4 = 1$ and
$(H^* - E)^3 \cdot E = 1$. Hence $\varphi$ is a birational morphism, 
$\varphi(E) = \PP^3$, and $\varphi(\widetilde W) = \PP^4$.
Since $R \sim H^* - kE$ for some $k \ge 2$, 
we have $(H^* - E)^3\cdot R = 2 - k \ge 0$. Thus $k = 2$ and 
$\dim \varphi(R) \le 2$. 
\end{proof}

The next corollary is straightforward. 

\begin{corollary}\label{Corollary-3.3.} 
In the notation as above, 
let $M\subset W$ be a hyperplane section containing the 
$\sigma_{2,2}$-plane $\Xi$ in $W$
\textup(the case $M=R$ is not excluded\textup),
and let $\widetilde M$ be the proper transform of $M$ in $\widetilde W$.
Then $\varphi(\widetilde M)$ is a hyperplane in $\PP^4$, and 
$W\setminus (M\cup R)\simeq \PP^4\setminus 
(\varphi(\widetilde M)\cup \varphi(E))$.
In particular, $W\setminus (M\cup R)$ contains a cylinder.
\end{corollary}

\begin{proposition}[{\cite[II, \S 10]{Fujita1980-84}}, {\cite{Prokhorov1994}}]
\label{Proposition 3.8.}
Let $\Pi\subset W$ be a $\sigma_{3,1}$-plane. Then 
there is a commutative diagram 
\[
\xymatrix{
&\widetilde W\ar[dr]^{\varphi}\ar[dl]_{\rho}&
\\
W\ar@{-->}[rr]^{\phi} &&Q\subset \PP^4 
}
\]
where 
\begin{enumerate}
\item 
$\rho : \widetilde W\longrightarrow W$ is the blow-up of 
$\Pi$, $Q\subset \PP^4$ is a smooth quadric, and
$\phi : W\dashrightarrow Q\subset \PP^4$ is 
the projection from $\Pi$; 

\item 
the morphism $\varphi :\widetilde W \longrightarrow Q\subset \PP^4$ 
is defined by the linear system $|\rho^*H - E|$, 
where $E =\rho^{-1}(\Pi)$ 
is the exceptional divisor of $\rho$; 

\item 
a general fiber of $\varphi$ is the strict transform of a line meeting $\Pi$, 
and each one-dimensional 
fiber of $\varphi$ is isomorphic to $\PP^1$; 

\item there exists 
exactly
one two-dimensional fiber $\widetilde \Xi$ of $\varphi$, 
where $\widetilde \Xi\simeq \PP^2$ is the strict transform of the
$\sigma_{2,2}$-plane $\Xi$ in $W$; 

\item the restriction $\varphi |_ E: E \to Q $
is the blow-up of a line $Y\subset Q$, and the image $\varphi(\widetilde \Xi)$ is a point in $Y$.
\end{enumerate}
\end{proposition}

\begin{proof} 
Since $\Pi$ is a scheme-theoretic intersection of members of the linear system $|H-\Pi|$,
the linear system $|\rho^*H-E|$ is base point free.
Hence the divisor $-K_{\widetilde W}=2\rho^*H+\rho^*H-E$ is ample, i.e. 
$\widetilde W$ is a Fano fourfold with $\rk \Pic(\widetilde W)=2$.
By the Cone Theorem there exists a Mori contraction $\varphi: \widetilde W\to U$ 
different from $\rho$.
If $ \widetilde C\subset \widetilde W$ is the proper transform of a line $C\subset W$ 
meeting $\Pi$,
then $(\rho^*H-E)\cdot \widetilde C=0$. In particular, the divisor $\rho^*H-E$ 
is not ample.
So, $\rho^*H-E$ defines a supporting function for the extremal ray generated 
by curves contained in 
the fibers of $\varphi$. Moreover, we can write $\rho^*H-E=\varphi^* L$, 
where $L$ is the ample generator of 
$\Pic(U)\simeq \ZZ$. 
We have $L^4=(\rho^*H-E)^4=0$.
By the Riemann-Roch and the Kodaira Vanishing Theorems we obtain
$\dim |\rho^*H-E | =4$.
Therefore, $|\rho^*H-E |$ defines a 
morphism, say, $\upsilon:\widetilde W\to Q\subset \PP^4$, 
where $\upsilon(E)=\upsilon(\widetilde W)$ is a 
quadric $Q\subset \PP^4$; indeed,
$(\rho^*H-E )^3\cdot E=2$.
The restriction $\upsilon|_E: E\to Q$ is a birational morphism. 
Furthermore,
$\upsilon$ can be factorized as follows: 
$\upsilon:\widetilde W \overset{\varphi}{\longrightarrow} U \longrightarrow Q$.
Since $Q$ is normal, the birational morphism $U \longrightarrow Q$ is an isomorphism. 
Hence we may identify $U$ with $Q$ and $\upsilon$ with $\varphi$.
Since $\rk\Pic(Q)=1$ and $\rk\Pic(E)=2$, 
the birational morphism $\varphi|_E:E\to Q$ cannot be an isomorphism. 

On the other hand, the Mori extremal contraction $\varphi$ generically is a 
$\PP^1$-bundle with 
at most finite number of fibers of dimension $2$.
Therefore $\dim \Sing(Q)\le 0$. The divisor $-K_E=(3\rho^*H-2E)|_E$ is ample, i.e. $E$ 
is a Fano threefold with 
$\rk\Pic(E)=2$. Then $\varphi|_E:E\to Q$ is as well an extremal Mori contraction. 
Since a quadric threefold with isolated singularity is not $\QQ$-factorial,
$Q$ must be smooth. This shows (i)-(iii).

From the classification of 
three-dimensional extremal Mori contractions we may conclude that the image, say, $Y$ of $E$ in 
$\PP^4$ is a smooth curve. We claim that $Y$ is in fact a line. Indeed, by the Adjunction Formula 
we have $K_E^3=(K_{\widetilde W}+E)^3\cdot E$. A direct computation of the both sides 
of the latter equality 
using the formulas in Lemmas \ref{lemma-intersection-theory}(ii) and \ref{Lemma-2.4.} 
gives $\deg Y=1$, as claimed. 

Let further $\widetilde \Xi_{1},\dots,\widetilde \Xi_{n}$ be the two-dimensional fibers of $\varphi$
(the case $n=0$ is not excluded).
Then each $\rho(\widetilde \Xi_{i})$ is a plane meeting $\Pi$ along a line. 
Since $\varphi$ is a $\PP^1$-bundle outside of $\widetilde \Xi_{1}\cup \dots\cup\widetilde \Xi_{n}$,
we have 
\[
\Eu(\widetilde W)=\Eu(W)+3 =9= \sum \Eu( \widetilde \Xi_{i})+2(\Eu(Q)-n)=8+n.
\]
Hence $n=1$, i.e. there exists exactly one two-dimensional fiber $\widetilde \Xi=\widetilde \Xi_1$.
By Proposition \ref{Proposition-2.2.} $\rho(\widetilde \Xi)$ is a $\sigma_{2,2}$-plane.
This proves (iv) and (v). 
\end{proof}

Now we can deduce the following corollary.

\begin{corollary}\label{Corollary-3.3a.}
In the notation as before, 
let $M\subset W$ be a hyperplane section containing a 
$\sigma_{3,1}$-plane $\Pi$ in $W$, 
and let $\widetilde M $ be its proper transform in $\widetilde W$.
Then $\varphi(\widetilde M)$ is a hyperplane section of $Q\subset \PP^4$ and 
$W\setminus M\simeq \widetilde W\setminus (\widetilde M\cup E)$.
The projection $\widetilde W\setminus (\widetilde M\cup E) \to Q\setminus \varphi(\widetilde M)$
is a locally trivial $\mathbb A^1$-bundle. 
In particular, $W\setminus (M\cup R)$ contains a cylinder.
\end{corollary}

\noindent \textit{Proof of Theorem \textup{\ref{theorem-w5}}}. 
By Lemma 3.5, the singular hyperplane section $A$ of $W$ as in 
Theorem \ref{theorem-w5} contains a plane. 
This is either a $\sigma_{2,2}$-plane, or a $\sigma_{3,1}$-plane. In any case, 
by Corollaries \ref{Corollary-3.3.}
and \ref{Corollary-3.3a.}
the complement $W\setminus A$ contains a cylinder. 
\qed

\section{Cylindrical Mukai fourfolds of genus 8} \label{sec:g=8}
In this section we construct a family of cylindrical Mukai fourfolds of genus 
$8$ in $\PP^{10}$, see 
Theorem \ref{thm:cylinder-in-V14} below.
According to \cite{Mukai-1989} any Fano fourfold of index $2$ and genus $8$ 
with $\rk \Pic(V)=1$ is a section of 
the Grassmannian $\Gr(2,6)$ under its Pl\"ucker embedding 
in $\PP^{14}$ by a linear subspace of dimension $10$. 
The Grassmannian $\Gr(2,6)$
contains planes of two kinds, namely, the two-dimensional Schubert 
varieties $\sigma_{3,3}$ and $\sigma_{4,2}$, where 
\begin{itemize}
 \item 
$\sigma_{3,3}=\{l\in \Gr(2,6)\,|\,l \subset \Xi\}$ with $\Xi\subset \PP^5$ a fixed
plane, and
\item 
$\sigma_{4,2}=\{l\in \Gr(2,6)\,|\,p\in l \subset \Lambda\}$ with $\Lambda\subset \PP^5$ a fixed 
linear 3-subspace, and $p\in\Lambda$ a fixed point.
\end{itemize}

\begin{lemma}\label{lem:V-14-plane}
Let $V=V_{14}$ Fano-Mukai fourfold of index $2$ and genus $8$. Suppose that $V$ 
contains a $\sigma_{4,2}$-plane $\Pi$.
Then $c_2(\NNN_{\Pi/V})=2$. 
\end{lemma}

\begin{proof}
Let $l$ be the class of a line on $\Pi$.
Likewise as in Corollary \ref{Chern-classes-W5}
we have
$c_1(V)|_\Pi=2\sigma_{1,0}|_\Pi=2l$ and
$c_2(V)|_\Pi= 2\sigma_{2,0}|_\Pi+4\sigma_{1,1}|_\Pi$. Then
similarly as in Lemma \ref{lemma-W5-normal-bundle} we obtain
\begin{eqnarray*}
c_1(\NNN_{\Pi/V})&=&c_1(V)|_\Pi-c_1(\Pi)=-l,
\\
c_2(\NNN_{\Pi/V})&=& c_2(V)\cdot\Pi-c_1(\NNN_{\Pi/V})\cdot c_1(\Pi)-c_2(\Pi)=
\\
&=&2\sigma_{2,0}\cdot\Pi+4\sigma_{1,1}\cdot\Pi+3-3=2.
\end{eqnarray*}
\end{proof}

\begin{lemma}
\label{lem:V-14} 
There exists a smooth section
$V=V_{14}$ of $\Gr(2,6)\subset\PP^{14}$ by a linear subspace 
$L\simeq\PP^{10}$ 
containing a $\sigma_{4,2}$-plane $\Pi$. Furthermore, such a section $V$ 
can be chosen so that 
$\Pi$ does not meet any other plane contained in $V$ along a line. 
\end{lemma}

\begin{proof} The first assertion follows immediately from Lemma \ref{lem:smooth section}. 
To show the second one, 
assume that $\Pi$ meets another plane $\Pi'\subset V$ along a line, and let $K\simeq\PP^3$ 
be the linear span of $\Pi\cup\Pi'$ in $L\simeq\PP^{10}$. 

We claim that if $\Pi'$ is a $\sigma_{4,2}$-plane, then $K$ is contained in 
the Grassmannian $\Gr(2,6)$, and hence also in $V$. 
The latter yields a contradiction because $\Pic(V)=\ZZ\cdot [H]$.
To show the claim, notice that $\Pi\cap\Pi'\simeq\PP^1$ consists of 
all lines in a plane $N\simeq \PP^2$ in $\PP^5$ passing through 
a given point $P$. This plane $N$ is the intersection of the two linear 
$3$-subspaces, say, $M$ and $M'$ in $\PP^5$ that define our Schubert 
varieties $\Pi$ and $\Pi'$, respectively. Let $R\simeq\PP^4$ 
be the linear span of $M\cup M'$ in $\PP^5$.
Consider the Schubert variety $S\simeq\PP^3$ in the Grassmannian 
$\Gr(2,6)$, which consists of all lines through $P$ contained in $R$. 
Its image under the Pl\"ucker embedding of $\Gr(2,6)$ in $\PP^{14}$ 
is a linear $3$-subspace containing $\Pi\cup\Pi'$. Hence this image coincides 
with $K$. This proves the claim.

The latter argument does not work in the case, where
$\Pi'$ is a $\sigma_{3,3}$-plane. However, 
this possibility can be ruled out as well by choosing carefully a section $L$ 
through $\Pi$. 

Indeed, let $G$ be the set of all linear subspaces of dimension 10 
in $\PP^{14}$ through the given $\sigma_{4,2}$-plane $\Pi$. Then 
$G\simeq \Gr(8,12)$, so $\dim G=32$. Consider further a 
$\sigma_{3,3}$-plane $\Pi'$ that meets $\Pi$ along a line. 
Then the plane in $\PP^5$ that corresponds to $\Pi'$ contains 
the point corresponding to $\Pi$ and is contained in the corresponding linear 
3-subspace in $\PP^5$. The set of all such planes in $\PP^5$ is 
two-dimensional, hence also the set of all such possible 
$\sigma_{3,3}$-planes
$\Pi'$ in $\PP^{14}$ is.

Fixing $\Pi'$ we consider the set $G'$ of all linear subspaces 
of dimension 10 in $\PP^{14}$ through the linear 3-space 
$K=\operatorname{span} (\Pi\cup \Pi')$. Then $G'\simeq \Gr(7,11)$, and so 
$\dim G'=28$. Finally, let $\mathcal{E}$ be the variety of all possible 
configurations $(\Pi', L)$ as before. Due to our observations we have 
$\dim\mathcal{E}\le
28+2=30<32=\dim G$. 
Hence a general section $V=L\cap \Gr(2,6)$ through $\Pi$ does not contain a 
$\sigma_{3,3}$-plane $\Pi'$ that meets $\Pi$ along a line.

\end{proof}

\begin{lemma}\label{lem:V14-codim} Let $\mathcal{V}$ be 
the family of all smooth fourfold linear sections $V_{14}$ 
of the Grassmannian $\Gr(2,6)$, and $\mathcal{V}_{4,2}$ be the subfamily 
of those sections
that contain a $\sigma_{4,2}$-plane. Then $\mathcal{V}_{4,2}$ has codimension $1$ 
in $\mathcal{V}$. 
\end{lemma}

\begin{proof} We keep the notation from the proof of Lemma \ref{lem:V-14}.
The variety 
$\mathcal{P}$ of all 
$\sigma_{4,2}$-planes in the Grassmannian $\Gr(2,6)$ is isomorphic to the 
variety 
of all flags $\pt\in \PP^3 \subset \PP^5$. 
The latter variety has dimension $11$. It follows that 
$\dim\mathcal{V}_{4,2}\le \dim G+\dim\mathcal{P}=32+11=43$.
Let us show that actually $\dim\mathcal{V}_{4,2}=43$. 

Indeed, consider the incidence variety 
$\mathscr{I}=\{(\Pi, L)\,|\,\Pi\subset L\}\subset \mathcal{P}
\times \mathcal{V}_{4,2}\,.$ 
We claim that the natural surjection
 $\pr_2: \mathscr{I}\to\mathcal{V}_{4,2}$ is generically finite,
 or, which is equivalent, that a generic member $V\in \mathcal{V}_{4,2}$ 
contains at most finite number of $\sigma_{4,2}$-planes $\Pi$.
 Assume that $\Pi$ belongs to a family of $\sigma_{4,2}$-planes $\Pi_t\subset V$.
By Lemma \ref{lem:V-14-plane} we have $\Pi\cdot \Pi_t=\Pi^2=2$.
Since $\Pi$ and $\Pi_t$ are planes, they cannot meet each other 
at two points. Hence $\Pi\cap \Pi_t$ is a line. On the other hand, 
it was shown 
in the proof of Lemma \ref{lem:V-14} that $\Pi_t$ and $\Pi$ cannot meet
along a line, a contradiction.
Hence a generic Mukai fourfold $V\in \mathcal{V}_{4,2}$ 
contains a finite number of $\sigma_{4,2}$-planes, as claimed.

Since the projection $\pr_1: \mathscr{I}\to \mathcal{P}$ 
is surjective, and its fiber $G$ over 
a given $\sigma_{4,2}$-plane $\Pi\in \mathcal{P}$ has dimension 11, we have 
$\dim \mathscr{I}=\dim G +\dim \mathcal{P}=43$. Furthermore, 
since the second projection $\pr_2: \mathscr{I}\to\mathcal{V}_{4,2}$ 
is surjective and generically one-to-one, we get
 $\dim\mathcal{V}_{4,2}=\mathscr{I}=43$, as desired. 
On the other hand, the variety $\mathcal{V}$ of all Mukai 
fourfolds $V_{14}$ in $\PP^{14}$ can be naturally identified with 
an open set in the Grassmannian $\Gr(11,15)$. Hence $\mathcal{V}$ is irreducible of dimension 
$\dim \Gr(11,15)=44$, and so, $\mathcal{V}_{4,2}$ has codimension 1 in 
$\mathcal{V}$. Now the assertion follows. 
\end{proof}

\begin{lemma}\label{Chern-classes-V14} We have
\[
c_1(V)=2\sigma_{1,0}|_V,\quad c_2(V)
=2\sigma_{2,0}|_V+4\sigma_{1,1}|_V,\quad\mbox{and}\quad \Eu(V)=12.
\]
\end{lemma}
\begin{proof}
The proof is similar to that of Corollary \ref{Chern-classes-W5}.
\end{proof}

\begin{proposition}\label{prop:W-5} 
Let $V=V_{14}\subset \PP^{10}$ be a Mukai fourfold of genus $8$ containing 
 a $\sigma_{4,2}$-plane $\Pi$ (see Lemma \ref{lem:V-14}). 
Then there is a commutative diagram 
\[
\xymatrix{
&\widetilde V\ar[dr]^{\varphi}\ar[dl]_{\rho}&
\\
V\ar@{-->}[rr]^{\phi} &&W
}
\]
where 
\begin{enumerate} 
\item 
$\rho : \widetilde V\longrightarrow V$ is the blow-up of $\Pi$ and
$\phi : V\dashrightarrow W\subset\PP^7$ is the projection with center $\Pi$, 
which sends $V$ birationally to a quintic Fano fourfold $W=W_5$ in $\PP^7$ 
with $\Pic (W)=\ZZ\cdot[\OOO_W(1)]$;

\item $\varphi : \widetilde V\longrightarrow W$ is 
the blowup of a smooth rational surface $F\subset W$ of degree $7$ 
with $-K_F\cdot\OOO_F(1)=5$, contained in a singular hyperplane section $L$ 
of $W$. 
This surface $F$ can be obtained by blowing up $6$ points in $\PP^2$;

\item 
$\varphi :\widetilde V \longrightarrow W\subset\PP^7$ is defined by the linear system 
$|\rho^* H - E|$ on $\widetilde V$, 
where $E =\rho^{-1}(\Pi)\subset \widetilde V $ 
is the exceptional divisor of $\rho$ and $H$ is a hyperplane in $\PP^{10}$; 

\item 
$\varphi(E) = L = \langle F\rangle$
is the linear span of $F$; 

\item 
if $D=\varphi^{-1}(F)$ is
the exceptional divisor of $\varphi$, then $D\sim L^* - E\sim H^* - 2E$ 
and $H^*\sim 2L^* - D$;

\item $V\setminus\rho(D)\simeq W\setminus \varphi(E)$, where $\rho(D)$ is cut out in $V$ by 
a quadric hypersurface in $\PP^{10}$ containing $\Pi$. 

\end{enumerate}
\end{proposition}

\begin{proof}
Since $\Pi$ is a scheme-theoretic intersection of members of the linear system $|H-\Pi|$,
the linear system $|\rho^*H-E|$ is base point free.
Hence the divisor $-K_{\widetilde V}=\rho^*H+\rho^*H-E$ is ample, i.e. 
$\widetilde V$ is a Fano fourfold with $\rk \Pic(\widetilde V)=2$.
By the Cone Theorem there exists a Mori contraction $\varphi: \widetilde V\to W$ different from $\rho$.
If $ \widetilde C\subset \widetilde W$ is the proper transform of a line $C\subset V$ meeting $\Pi$,
then $(\rho^*H-E)\cdot \widetilde C=0$. In particular, the divisor $\rho^*H-E$ is not ample.
So, $\rho^*H-E$ defines a supporting function for the extremal ray generated 
by curves contained in 
the fibers of $\varphi$. Moreover, we can write $\rho^*H-E=\varphi^* L$, 
where $L$ is the ample generator of 
$\Pic(W)\simeq \ZZ$. 
We have $L^4=(\rho^*H-E)^4=5$.
By the Riemann-Roch and Kodaira Vanishing Theorems we have
$\dim |\rho^*H-E | =7$.
Therefore, $|\rho^*H-E |$ defines a birational
morphism $\upsilon:\widetilde V\to W'\subset \PP^7$.
Further, $\dim |\rho^*H-2E | =0$ and $(\rho^*H-E)^3\cdot (\rho^*H-3E)=0$.
Thus $\varphi$ contracts a unique divisor $D\in|\rho^*H-2E |$. Since $\rk \Pic (\widetilde W)=2$, 
the divisor
$D$ is irreducible. 
Furthermore, there is a commutative diagram
\[
\xymatrix{
&\widetilde V\ar[dl]_{\rho}\ar[r]^{\varphi}\ar[dr]^{\upsilon}&W\ar[d]
\\
V\ar@{-->}[rr]^{\phi}&&W'
} 
\]
where the map $v:\widetilde V\to W'\subset \PP^7$ is given by the linear system $|\rho^*H-E|$ 
and 
$\upsilon: \widetilde V\overset{\varphi}{\longrightarrow} W\longrightarrow W'$
is the Stein factorization.

Since $(\rho^*H-E)^2\cdot D^2=-7$, the image $F=\varphi(D)$ is a surface in $W$
with $L^2\cdot F=7$.
Note that $E\simeq \PP_{\PP^2}(\NNN_{\Pi/V}^*)$.
Let $B\subset \widetilde V$ is a two-dimensional fiber of $\varphi$ not contained in $E$.
Then by our construction $\rho(B)$ is a plane meeting $\Pi$
along a conic. This contradicts our assumption.

Assume further that there is a two-dimensional fiber $B\subset \widetilde V$ of $\varphi$ contained in $E$.
Then $\upsilon(E)$ is a cone. Indeed, the images of 
the fibers of $E\to \Pi$ are lines in $W'\subset \PP^7$ passing through the point $\upsilon(B)$.
Moreover, $\upsilon(E)$ is a cone over a surface which is an image of $\PP^2$.
Since $\deg \upsilon(E)= (\rho^*H-E)^3\cdot E=5$, we get a contradiction.
Therefore, all fibers of $\varphi$ have dimension $\le 1$, the 
contraction $E\to \varphi(E)$ is small, and the variety 
$\varphi(E)$ is a del Pezzo threefold with isolated singularities.
(A description of such threefolds can be found in \cite[5.3.5]{Prokhorov-GFano-1}.)
By \cite{Ando1985} both $V$ and $F$ are smooth and $\varphi$ is the 
blowup of $F$. Since $-K_W=\varphi_*(-K_{\widetilde V})=3L$, $W$ is a Fano 
fourfold of index $3$ and degree $L^4=5$.
By the classification $L=-\frac13 K_V$ is very ample, so $V'\to V$ is an isomorphism.
Since the divisor $E\sim \varphi^*L-D$ is not movable, 
$\langle F\rangle=\PP^6$.

Finally, using Lemma \ref{lemma-intersection-theory} one can deduce that
$L\cdot (-K_{F})= 5$.
So, a general hyperplane section of $F$ is a smooth curve of genus $2$,
and $F$ is a surface of negative Kodaira dimension, i.e. $F$ is birationally ruled.
For the Euler numbers we obtain $\Eu(V)=12$ and $\Eu(W)=6$ (see Lemma \ref{Chern-classes-W5} 
and Corollary \ref{Chern-classes-V14}). So
by our construction $\Eu(F)=9$. 
Let $M\subset F$ be a general hyperplane section. If the divisor $K_F+M$ is not nef, then 
there exists an
extremal ray $R$ such that $(K_F+M)\cdot R<0$. Since $M$ is ample, $R$ cannot be
generated by a $(-1)$-curve. Hence $F$ is a geometrically ruled surface, and 
$R$ is generated by its rulings. In the latter case the Euler number $\Eu (F)$
must be even, a contradiction.

Thus the divisor $K_F+M$ is nef. This yields the inequalities $0\le (K_F+M)^2=K_F^2-3$, $K_F^2\ge 3$,
and so $F$ is a rational surface.
By Noether formula $K_F^2=3$ and $\rk \Pic(F)=7$.
By Riemann-Roch and Kodaira Vanishing, $\dim H^0(F, K_F+M)=2$.
Since $(K_F+M)^2=0$, the linear system $|K_F+M|$ is a base point free pencil.
It defines a morphism $\Phi_{|K_F+M|}: F\to \PP^1$ such that $-K_F$ is relatively ample. 
Hence $\Phi_{|K_F+M|}$ is a conic bundle with $5=\rk \Pic(F)-2$ degenerate fibers.
Let $\Sigma\subset F$ be a section of this bundle 
with the minimal possible self-intersection number
$\Sigma^2=-n$. Then $K_F\cdot \Sigma=n-2$, $1=(K_F+M)\cdot \Sigma=M\cdot \Sigma+n-2$,
and so $n=3-M\cdot \Sigma\le 2$. On the other hand, $n\ge 1$. It is possible to
 contract extra components of the 5 degenerate fibers of $F$ in order to
get a relatively minimal rational ruled surface $F'$ with a section $\Sigma'\subset F'$ such that 
${\Sigma'}^2=-1$, i.e. $F'\simeq \FF_1$. This means that $F$ can be obtained by 
blowing up $6$ points on $\PP^2$.
\end{proof}

We can deduce now the main result of this section.

\begin{theorem}\label{thm:cylinder-in-V14} 
Let $V=V_{14}$ be the Mukai fourfold 
of genus $8$ constructed in Lemma \textup{\ref{lem:V-14}}, and $\rho(D)$ be the divisor on $V$ 
constructed in Proposition \textup{\ref{prop:W-5}}.
Then
the Zariski open set $V\setminus \rho(D)$ contains a cylinder, and so $V$ 
is cylindrical.
\end{theorem}

\begin{proof} Indeed, by (iv) and (vi) in Proposition \ref{prop:W-5} we have $V\setminus\rho(D)\simeq 
W\setminus \varphi(E)\simeq W\setminus L$. 
Using Theorem \ref{theorem-w5} with $A=L$, the result follows.\end{proof}

\section{Cylindrical Mukai fourfolds of genus 7}\label{sec:g=7}
\label{section-g=12}
In this section we prove the following theorem. 

\begin{theorem}\label{thm:cylinders-in-V12} 
There exists a family of smooth cylindrical Mukai fourfolds $V=V_{12}\subset \PP^{9}$
of genus $7$
with $\Pic(V)\simeq \ZZ$. Its image in the corresponding moduli space has codimension $2$. 
\end{theorem}

The proof exploits several auxiliary results. 
Actually, our Mukai fourfold $V$ is obtained starting with a del Pezzo fourfold $W_{2\cdot 2}$ 
via a Sarkisov link, 
as described in Proposition \ref{thm:Sarkisov-link-V12} below.
We chose this link in such a way that the cylinder structure is preserved. 
This is the main point of our construction. 

\begin{proposition}
[\cite{Prokhorov-1993c}]
\label{thm:Sarkisov-link-V12}
Let $W=W_{2\cdot 2}\subset \PP^6$ be a smooth intersection of two quadrics,
and $H$ be a hyperplane section of $W$, so that the class of $H$ 
is the ample generator of $\Pic(W)$.
Suppose that $W$ contains an anticanonically embedded del Pezzo 
surface $F=F_5$ of degree $5$ and 
does not contain any plane which meets $F$ along a conic.
Then the following hold.
\begin{enumerate}\item
The linear system $|2H-F|$ of quadrics passing through $F$ 
defines a birational map $\phi: W \dashrightarrow V=V_{12}\subset \PP^{9}$,
where $V=\phi(W)$ is a Mukai fourfold of genus $7$
with $\Pic(V)\simeq \ZZ$. \item There is a commutative diagram
\begin{equation}
\label{equation-V12-diagram}
\xymatrix{
&\widetilde W\ar[dr]^{\varphi}\ar[dl]_{\rho}&
\\
W\ar@{-->}[rr]^{\phi}&&V 
} 
\end{equation}
where $\rho$ is the blowup of $F$ and $\varphi$ is the blowup of 
a plane $\Xi\subset V\subset \PP^9$.
\item
Let $E\subset\widetilde W$ \textup($D\subset\widetilde W$, respectively\textup) 
be the $\rho$-exceptional \textup($\varphi$-exceptional, respectively\textup) divisor, and 
let $L$ be the ample generator of $\Pic(V)$.
Then 
\[
\varphi^*L\sim 2\rho^*H-E,\quad D\sim \rho^*H-E,\quad \rho^*H\sim \varphi^*L -D,\quad 
E\sim \varphi^*L -2D. 
\]
\end{enumerate}
\end{proposition} 

\begin{proof} 
Since $F$ is a scheme-theoretic intersection of quadrics,
the linear system $|2\rho^*H-E|$ is base point free.
Hence the divisor $-K_{\widetilde W}=\rho^*H+2\rho^*H-E$ is ample, i.e. 
$\widetilde W$ is a Fano fourfold with $\rk \Pic(\widetilde W)=2$.
By the Cone Theorem there exists a Mori contraction $\varphi: \widetilde W\to U$ different from $\rho$.
Let $H_F$ be the hyperplane section of $W$ that passes through $F$,
and let $D$ be its proper transform in $\widetilde W$. 
We can write $D\sim \rho^*H-kE$ for some $k>0$. On the other hand, we have 
$0\le (2\rho^*H-E)^3\cdot D= -12(k-1)$. Hence, $k=1$ and $(2\rho^*H-E)^3\cdot D=0$.
This means that the divisor class of $2\rho^*H-E$ is not ample, and so it yields a supporting linear 
function of the extremal ray generated by the curves in the
fibers of $\varphi$. 
Moreover, we can write $2\rho^*H-E=\varphi^* L$, where $L$ is the ample generator of 
$\Pic(V)\simeq \ZZ$. 
We have $L^4=(2\rho^*H-E)^4=12$. It follows that $\dim V=4$, i.e. 
$\varphi$ is birational, and its exceptional locus, say, $D$
is an irreducible divisor. 
Using the Riemann-Roch and Kodaira Vanishing Theorems we obtain the equality
$\dim |2\rho^*H-E | =9$. 
This yields a diagram
\[
\xymatrix{
&\widetilde W\ar[d]_{\rho}\ar[r]^{\varphi}&V\ar[d]
\\
&W\ar@{-->}[r]^{\phi}&V'
} 
\]
where $\widetilde W\to V'\subset \PP^9$ is given by the linear system $|2\rho^*H-E|$,
and $\widetilde W\overset{\varphi}{\longrightarrow} V\longrightarrow V'$
is the Stein factorization.

Since $(2\rho^*H-E)^2\cdot D^2=-1$, then $\varphi(D)$ is a surface 
with $L^2\cdot\varphi(D)=1$.
Let $B\subset \widetilde W$ be a two-dimensional fiber of $\varphi$ not contained in $E$.
Then by our geometric construction, $\rho(B)$ is a plane meeting $F$
along a conic. This contradicts our assumption.
If $B\subset \widetilde W$ is a two-dimensional fiber of $\varphi$ contained in $E$, 
then the restriction $\rho|_B: B \to F$ is a finite morphism.
Hence $\rk \Pic (B)\ge\rk \Pic (F)=5$.
This contradicts the classification of fourfold contractions \cite{Andreatta1998a}.
Therefore, all fibers of $\varphi$ have dimension $\le 1$.
By \cite{Ando1985}, both $V$ and $\varphi(D)$ are smooth, and $\varphi$ is the 
blowup of $\varphi(D)$. Since $-K_V=\varphi_*(-K_W)=2L$, the variety $V$ is a Fano 
fourfold of index $2$ and of genus $g=\frac12 L^4+1=7$.
By virtue of the classification, $L=-\frac12 K_V$ is very ample, so $V'\to V$ is an isomorphism.
\end{proof}

The following corollary is immediate.

\begin{corollary}\label{cor:first-isomorphism}
We have $V\setminus \varphi(E)\simeq W\setminus \rho(D)$, where $\varphi(E)$ and $\rho(D)$ are 
 hyperplane sections of $V$ and $W$, respectively. Moreover, $\varphi(E)$ is singular along 
$\Xi$, and $\rho(D) =W\cap \langle F\rangle$, where $\langle F\rangle\simeq\PP^5$ 
is the linear span of $F$
in $\PP^6$.
\end{corollary}

\begin{corollary}\label{corrollary-V12-} Let
$V$ be a variety as in Proposition \ref{thm:Sarkisov-link-V12}.
Then the number of planes contained in $V$ is finite, 
and the group $\Aut(V)$ is finite as well. 
\end{corollary}
\begin{proof}
Assume that $V$ contains a family of planes $\Xi_t$.
The map $\phi^{-1}$ is a projection from $\Xi$. 
Hence the image of a general plane $\Xi_t$ is again a plane.
On the other hand, the set of planes contained in $W=W_{2\cdot 2}\subset \PP^6$ 
is finite  \cite{Reid1972}.

The group $\Aut(V)$ consists of the 
projective transformations of $\PP^{9}$ preserving $V$, 
so this is a linear algebraic group. Since the number of planes contained in $V$ is finite,
the identity component $\Aut^0(V)$ preserves each of these planes. 
In particular, it preserves the center $\Xi$ of the blowup $\varphi$. Hence  diagram 
\eqref{equation-V12-diagram} is 
$\Aut^0(V)$-equivariant with respect to a faithful
$\Aut^0(V)$-action on $W$.

However, the group $\Aut^0(W)$ is trivial. Indeed, the embedding $W\hookrightarrow\PP^6$ 
being given by the linear system $|-\frac13 K_W|$, 
 the latter group acts linearly on $\PP^6$ and preserves  every degenerate member 
of the pencil of quadrics in $\PP^6$
passing through $W$. These degenerate members are 7 quadric cones, whose vertices are 
points in $\PP^6$ in general position fixed under the $\Aut^0(W)$-action.  
Hence the group $\Aut^0(V)$ is also trivial, and so $\Aut(V)$ is finite. 
\end{proof}

Due to Corollary \ref{cor:first-isomorphism}, to prove Theorem \ref{thm:cylinders-in-V12} 
it suffices to show the existence of a cylinder in $W\setminus H_0$,
where $H_0=\rho(D)$ is a hyperplane section of $W=W_{2\cdot 2}\subset \PP^6$ which
contains a quintic del Pezzo surface $F$. To this end, we apply Corollary \ref{corollary-cylinder-W2.2}.
The assumptions of Corollary \ref{corollary-cylinder-W2.2} are satisfied once 
there is a plane $\Pi_0\subset H_0$ which does not meet $F$ along a conic,
see Proposition \ref{thm:Sarkisov-link-V12}. 

 Using the following construction we produce examples, 
where these geometric restrictions are fulfilled. This gives 
the first part of Theorem \ref{thm:cylinders-in-V12}. 

\begin{sit}\label{sit:construction} 
{\bf Construction} (cf. \cite[5.3.9]{Prokhorov-GFano-1}).
Let $X=X_5\subset \PP^6$ be a Fano threefold of index 2 and of degree 5.
It is well known (see e.g.
\cite[Theorem 3.3.1]{Iskovskikh-Prokhorov-1999})
that $X$ can be realized as a section of 
the Grassmannian $\Gr(2,5)$ under its Pl\"ucker embedding in $\PP^9$
by a subspace of codimension 3. 
The projection from a general point $P'\in X_5$ 
sends $X_5$ to a singular Fano threefold $Y=Y_4\subset\PP^5$ of degree 4; the latter threefold 
 is a complete intersection of two quadrics, say, $Q_1'$ and $Q_2'$
(see \cite[Cor. 0.8]{Shin1989}). 
\end{sit}

\begin{lemma}\label{lem:add-1}
The variety Y constructed above 
contains an anticanonically embedded del Pezzo surface
$F =F_5\subset \PP^5$ of degree $5$ and a unique plane $\Pi_0$. This plane meets $F$ at three 
points, say, $A_j$, $j=1,2,3$, that are the only 
singular points of $Y$ and these singularities are ordinary double points.
\end{lemma}

\begin{proof}
By \cite{Furushima-1989a}, through a general point $P'\in X_5$ pass exactly 3 lines, 
say, $l_j$, $j=1,2,3$, on $X_5$. These lines do not 
belong to the same plane. Under our projection, they are contracted 
 to 3 distinct non-collinear points, say, $A_j$, $j=1,2,3$, of $Y_4$. 
By van der Waerden's purity theorem, 
the points $A_j$ are the only singularities (actually, nodes) of $Y_4$. 
Let $\Pi_0$ be the plane in $\PP^5$ through the points $A_j$, $j=1,2,3$. We claim that
$\Pi_0$ is contained in $Y_4$ and is a unique plane contained in $Y_4$. Indeed, 
consider the diagram 
\[
\xymatrix{
&\widetilde{X_5}\ar[dr]^{\varphi'}\ar[dl]_{\rho'}&
\\
X_5\ar@{-->}[rr]^{\phi'}&&Y_4
} 
\]
where $\rho'$ is the blowup of $P'$ and $\varphi'$ is the (small) contraction of the proper transforms 
$\widetilde{l_j}\subset\widetilde{X_5}$ of the lines $l_j$, $j=1,2,3$. Clearly, $\Pi_0=\varphi'(E')$, 
where $E'\simeq\PP^2$ 
is the exceptional divisor of the blowup $\rho'$. Hence $\Pi_0\subset Y_4$. 

We have 
$\rk \Pic (\widetilde{X_5})=2$, and by duality 
$\rk {\operatorname{N}_1(\widetilde{X_5}})_{\RR}=2$. 
For a general hyperplane section $H$ of 
$Y_4$ we have ${\varphi'}^*H\cdot\widetilde{l_j}=0$, $j=1,2,3$. It follows that for any plane 
$\Pi'$ contained in $Y_4$, the intersection numbers $\widetilde{\Pi'}\cdot \widetilde{l_j}$, $j=1,2,3$ are 
simultaneously all zero or not, where $\widetilde{\Pi'}$ is the proper transform of $ \Pi'$ in $\widetilde{X_5}$. 
If $\widetilde{\Pi'}$ does not meet the curves $\widetilde{l_j}$, then $ \Pi'$ does not pass through the singular 
points of $Y_5$ and so is a Cartier divisor on $Y_4$. Then $1=\deg\Pi'\equiv 0 \mod 4$, a contradiction. 
Thus $A_j\in\Pi'$, $j=1,2,3$, hence $\Pi'=\Pi_0$, as claimed. 

A general hyperplane section $F'$ of $X_5$ in $\PP^6$ is a smooth del Pezzo surface of degree 5. Since $F'$ 
meets transversally the lines $l_j$ and does not contain $P'$, it maps under $\phi'$ isomorphically onto its image, 
say, $F$ in $Y_4$, and 
$F\cap \Pi_0=\{A_1,A_2,A_3\}$.
 \end{proof}

\begin{sit}\label{sit:construction-2} 
Let $q_i(x_0,\ldots,x_5)=0$ be the equation of $Q_i'$ in $\PP^5$, $i=1,2$. Consider the quadrics $Q_i$ in $\PP^6$ 
with equations $q_i(x_0,\ldots,x_5)+x_6f_i(x_0,\ldots,x_5)=0$, $i=1,2$, where $f_1$ and $f_2$ are generic linear forms. 
We claim that the fourfold $W=W_{2\cdot 2}=Q_1\cap Q_2$ in $\PP^6$ is smooth and satisfies 
all the assumptions of Proposition \ref{thm:Sarkisov-link-V12}. To show the claim, let us notice that the hyperplane $x_6=0$ in $\PP^6$ cuts the quadric 
$Q_i$ along $Q_i'$ and cuts $W$ along $Y_4=Q_1'\cap Q_2'$. Hence $W$ contains the smooth del Pezzo surface $F\subset Y_4$ 
of degree 5, and does not contain any plane which meets $F$ along a conic. Indeed, otherwise such a plane would be contained 
in the hyperplane $x_6=0$, and so coincides with $\Pi_0$ by virtue of Lemma \ref{lem:add-1}. Since $\Pi_0$ meets $Y_4$ just in the points $A_j$, $j=1,2,3$, 
we get a contradiction. This proves the claim. 
\end{sit}

In the following lemma we provide an alternative construction of the family of pairs $(Y,F)$ as in 
\ref{sit:construction}-\ref{lem:add-1}, which will be used later on. 

\begin{lemma}
Let $F\subset \PP^5$ be a del Pezzo surface of degree $5$ and 
 $Q_1$, $Q_2$ be general quadrics in $\PP^5$ containing $F$.
Then $Y:=Q_1\cap Q_2$ is a threefold as in \textup{\ref{sit:construction}}. 
In particular, $Y$ contains a plane meeting $F$ in three points.
\end{lemma}

\begin{proof}
Let $Q\subset \PP^5$ be another general quadric containing $F$.
Then $Y\cap Q= F\cup F'$, where $F'$ is a cubic surface scroll.
The linear span $\Lambda=\langle F'\rangle$ is a subspace of dimension 4.
Hence $Y\cap \Lambda= F'\cup \Pi_0$, where $\Pi_0$ is a plane contained in $Y$. By \ref{sit:construction},
for a general choice of $Q_1$, $Q_2$, and $Q$, 
the variety $Y$ has only isolated singularities, 
and these singularities are nodes. Moreover, $Y$ contains no planes other than $\Pi_0$.
Such varieties $Y$ are described in \cite[5.3.9]{Prokhorov-GFano-1},
and their construction coincides with that of \ref{sit:construction}.
\end{proof}

 Using \cite{Mukai-1989} one can deduce that the moduli space of
the Mukai fourfolds of genus $7$ 
has dimension $15$. The second assertion of Theorem \textup{\ref{thm:cylinders-in-V12}} 
follows now from the next lemma. 

\begin{lemma}\label{lem:dimension-count}
The image in the moduli space of the family of all Fano fourfolds of genus $7$ 
obtained by our construction \textup{\ref{sit:construction}} has dimension $13$.
\end{lemma} 

\begin{proof} Recall that $F\subset \PP^5$ is an intersection of 5 linearly independent quadrics 
(\cite[Corollary 8.5.2]{Dolgachev-ClassicalAlgGeom}).
Thus the space of all quadrics in $\PP^6$ passing through $F$ 
has dimension $5+7=12$.
Pencils of quadrics passing through $F$ are parametrized by the Grassmannian $\Gr(2,12)$.
Since the group $\Aut (F)$ is finite, and any automorphism of $\PP^6$, which acts trivially on $F$, 
acts also trivially on $\PP^5$, the algebraic group $\Aut (\PP^6,F)=\Aut (\PP^6,\PP^5)$ 
has dimension 6, while $\dim\Aut (\mathbb{A}^7,\mathbb{A}^6)=7$. 
Modulo the $\Aut (\mathbb{A}^7,\mathbb{A}^6)$-action on $\Gr(2,12)$, we have 
$20- 7=13$-dimensional family of such pencils of quadrics.
Hence the dimension of the family of all Fano fourfolds that
can be obtained by our construction equals $13$. Its image in the moduli 
has the same dimension due to Corollary \ref{corrollary-V12-}.
\end{proof}

\end{document}